\documentclass[11pt]{amsart}

\usepackage{epigamath}

\usepackage[english]{babel}

\numberwithin{equation}{section}

\usepackage{enumitem}
\usepackage[all]{xy}

\newtheorem{thm}{Theorem}[section]

\newtheorem{lem}[thm]{Lemma}

\theoremstyle{definition}
\newtheorem{defn}[thm]{Definition}
\newtheorem{rem}[thm]{Remark}

\newtheorem*{b-divisor}{b-divisors} 
 
\newtheorem*{g-pair}{Generalized pairs} 
\newtheorem*{adj-g-pair}{Divisorial adjunction for generalized pairs} 
\newtheorem*{mmp-g-pair}{MMP for generalized pairs}

\newtheorem*{claim*}{Claim}

\def\restriction#1#2{\mathchoice
              {\setbox1\hbox{${\displaystyle #1}_{\scriptstyle #2}$}
              \restrictionaux{#1}{#2}}
              {\setbox1\hbox{${\textstyle #1}_{\scriptstyle #2}$}
              \restrictionaux{#1}{#2}}
              {\setbox1\hbox{${\scriptstyle #1}_{\scriptscriptstyle #2}$}
              \restrictionaux{#1}{#2}}
              {\setbox1\hbox{${\scriptscriptstyle #1}_{\scriptscriptstyle #2}$}
              \restrictionaux{#1}{#2}}}
\def\restrictionaux#1#2{{#1\,\smash{\vrule height .8\ht1 depth .85\dp1}}_{\,#2}}

\EpigaVolumeYear{5}{2021} \EpigaArticleNr{18} \ReceivedOn{June 28,
2021}
\InFinalFormOn{October 12, 2021}
\AcceptedOn{November 2, 2020}

\title{Crepant semi-divisorial log-terminal model}
\titlemark{Crepant semi-divisorial log-terminal model}

\author{Kenta Hashizume}
\address{Graduate School of Mathematical Sciences, 
The University of Tokyo, 3-8-1 Komaba Meguro-ku Tokyo 153-8914, Japan}
\email{hkenta@ms.u-tokyo.ac.jp, hashizumekenta1@gmail.com}

\authormark{K.~Hashizume}

\AbstractInEnglish{We prove the existence of a crepant sdlt model for slc pairs whose irreducible components are normal in codimension one. }



\TitleInFrench{Mod\`ele log-terminal semi-divisoriel cr\'epant}

\AbstractInFrench{Nous montrons l'existence d'un mod\`ele sdlt cr\'epant pour les paires slc dont les composantes irr\'eductibles sont normales en codimension 1.}

\acknowledgement{The author was partially supported by JSPS KAKENHI Grant Number JP19J00046.}

\begin{document}

%%%%%%%%%%%%%%%%%%%%%%%%%%%%%%%
% Title page
%%%%%%%%%%%%%%%%%%%%%%%%%%%%%%%

%\removeabove{}
%\removebetween{}
%\removebelow{}

\maketitle

\begin{prelims}

\DisplayAbstractInEnglish

\bigskip

\DisplayKeyWords

\medskip

\DisplayMSCclass

\bigskip

\languagesection{Fran\c{c}ais}

\bigskip

\DisplayTitleInFrench

\medskip

\DisplayAbstractInFrench

\end{prelims}

%%%%%%%%%%%%%%%%%%%%%
% Table of Contents
%%%%%%%%%%%%%%%%%%%%%

\newpage

\setcounter{tocdepth}{1}

\tableofcontents

%%%%%%%%%%%%%%%%%%%%%
% Content begins here
%%%%%%%%%%%%%%%%%%%%%

\section{Introduction}\label{intro}
Throughout this note, we will work over the complex number field $\mathbb{C}$. In this paper, we deal with semi-divisorial log terminal (sdlt, for short) pairs. 

\begin{defn}[Semi-divisorial log terminal pair, \protect{\cite[Definition 5.19]{kollar-mmp}}]\label{defn--sdlt}
An slc pair $(X,\Delta)$ in Definition \ref{defn--slc} is a {\em semi-divisorial log terminal} ({\em sdlt}, for short) {\em pair} if we put $U\subset X$ as the largest open subset such that $(U,\restriction{\Delta}{U})$ has only simple normal crossing singularities, then $a(E,X,\Delta)>-1$ for every prime divisor $E$ over $X$ whose image on $X$ is contained in $X\setminus U$. 
\end{defn}

As discussed in Section \ref{sec5}, sdlt pairs have much better structures than structures of slc pairs. 
In this note, we prove the following theorem. 

\begin{thm}[Crepant sdlt model]\label{thm--crepant-sdlt}
Let $(X,\Delta)$ be a quasi-projective slc pair such that $\Delta$ is a $\mathbb{Q}$-divisor and every irreducible component of $X$ is normal in codimension one. 
Then there is an sdlt pair $(Y,\Gamma)$ with a projective birational morphism $f\colon Y\to X$ satisfying the following properties.
\begin{enumerate}[label=(\roman*)]
\item
There is an open dense subset $U \subset X$ such that $U$ (resp.~$f^{-1}(U))$ contains all the codimension one singular points of $X$ (resp.~$Y$) and $f$ is an isomorphism over $U$, 
\item
$K_{Y}+\Gamma=f^{*}(K_{X}+\Delta)$, and 
\item
$K_{Y}$ is $\mathbb{Q}$-Cartier and $(Y,0)$ is semi-terminal in the sense of Fujita \cite{fujita}. 
\end{enumerate}
In particular, every irreducible component $Y_{j}$ of $Y$ is normal, and every lc pair $(Y_{j},\Gamma_{j})$ is dlt, where the $\mathbb{Q}$-divisor $\Gamma_{j}$ on $Y_{j}$ is defined by $K_{Y_{j}}+\Gamma_{j}=\restriction{(K_{Y}+\Gamma)}{Y_{j}}$.
\end{thm}

See \cite[Definition 2.3 (2)]{fujita} or Definition \ref{defn--semi-ter} for the definition of semi-terminal pairs.

\begin{rem}\label{rem--crepant-sdlt-difference}
Let $f \colon (Y,\Gamma)\to (X,\Delta)$ be a crepant sdlt model as in Theorem \ref{thm--crepant-sdlt}. 
Since $X$ is an S$_{2}$ scheme and $f$ is an isomorphism over all the generic points of codimension one singular locus, the morphism
$$\mathcal{O}_{X}(\lfloor m(K_{X}+\Delta)\rfloor) \longrightarrow f_{*}\mathcal{O}_{Y}(\lfloor m(K_{Y}+\Gamma)\rfloor)$$
is an isomorphism for every positive integer $m$. 
\end{rem}

\begin{rem}[Natural double cover by Koll\'ar \protect{\cite[5.23]{kollar-mmp}}, see also \protect{\cite[Lemma 5.1]{fujino-fund-slc}}]\label{rem--double-cover}
Let $(X,\Delta)$ be an slc pair such that $\Delta$ is an $\mathbb{R}$-divisor. 
Although irreducible components of $X$ may not be normal in codimension one, by  \cite[5.23]{kollar-mmp}, we can construct a quasi-\'etale double cover $\pi\colon X'\to X$ and an slc pair $(X',\Delta')$ such that $K_{X'}+\Delta'=\pi^{*}(K_{X}+\Delta)$ and every irreducible component of $X'$ is normal in codimension one. 
More precisely, applying \cite[5.23]{kollar-mmp} we get a finite morphism $\pi \colon X'\to X$ of degree two such that 
\begin{enumerate}[label=(\alph*)]
\item \label{rem--double-cover-(a)}
$X'$ is an S$_{2}$ scheme, 
\item \label{rem--double-cover-(b)}
$\pi$ is \'etale in codimension one, 
\item \label{rem--double-cover-(c)}
every irreducible component of $X'$ is normal in codimension one, and
\item \label{rem--double-cover-(d)}
the normalization of $X'$ is a disjoint union of two copies of the normalization of $X$. 
\end{enumerate}
Defining $\Delta'$ on $X'$ by $K_{X'}+\Delta'=\pi^{*}(K_{X}+\Delta)$, then the conditions  \ref{rem--double-cover-(a)}--\ref{rem--double-cover-(d)} imply that $(X',\Delta')$ is an slc pair whose irreducible components are normal in codimension one. If $X$ is quasi-projective and $\Delta$ is a $\mathbb{Q}$-divisor, then so is the above $X'$ and $\Delta'$.  
Thus, we can apply Theorem~\ref{thm--crepant-sdlt} to $(X',\Delta')$. 
\end{rem}

The assumption on the normality in codimension one for all irreducible components of $X$ looks artificial. 
But, all slc pairs admitting a crepant sdlt model satisfy this property because crepant models as in Theorem~\ref{thm--crepant-sdlt} do not modify any codimension one singular point and every irreducible component of the underlying schemes of sdlt pairs is normal. 
The first condition on $f$ in Theorem \ref{thm--crepant-sdlt} is important to preserve information of $K_{X}+\Delta$. 
We can not deduce properties of $K_{X}+\Delta$ from properties of $K_{Y}+\Gamma$ without the first condition in Theorem \ref{thm--crepant-sdlt} even in the case of the normalization of $X$. 
For example, Koll\'ar \cite{kollar-two-example} gave an example of projective slc surface whose canonical ring is not finitely generated, though it is well known that the log canonical rings of lc surfaces are finitely generated. 

In the case of lc pairs, a lot of partial resolutions are known. 
In the non-normal case, Koll\'ar 
and Shepherd--Barron \cite{ksb} established the minimal 
semi-resolution for surfaces, and Fujita \cite{fujita} 
established the semi-terminal modifications for demi-normal schemes. 
In \cite{odaka-xu}, Odaka and Xu proved the existence of an slc modification of demi-normal pairs under a certain assumption, and their result was generalized by Fujino and the author \cite{fh-lcmodel}.  
On the other hand, to the best of the author's knowledge, no result on the existence of a good crepant model of slc pairs is known. 
Theorem \ref{thm--crepant-sdlt} was motivated by \cite[Question 5.22.1]{kollar-mmp} asked by Koll\'ar. 
In \cite[Question 5.22.1]{kollar-mmp}, he asked whether there is a crepant sdlt model $(Y,\Gamma)\to (X,\Delta)$ as in Theorem \ref{thm--crepant-sdlt} such that the demi-normal pair $(Y,0)$ is not necessarily semi-terminal or $\mathbb{Q}$-Gorenstein but the coefficients of the exceptional divisors in $\Gamma$ are one. 
We could not give the complete affirmative answer to \cite[Question 5.22.1]{kollar-mmp}, but we hope that the main result will contribute developments of the theory of slc pairs. 

The contents of this note are as follows: in Section \ref{sec2}, we collect definitions of demi-normal pairs, simple normal crossing singularity, slc pair, semi-terminal pair, and an MMP step for slc pairs. 
In Section \ref{sec3}, we discuss a sufficient condition to construct a sequence of MMP steps for slc pairs. 
In Section \ref{sec4}, we prove Theorem~\ref{thm--crepant-sdlt}. 
In Section \ref{sec5}, we collect properties of sdlt pairs.

\subsection*{Acknowledgments}
The author is grateful to Professors J\'anos Koll\'ar and Yuji Odaka for discussions and comments. 
He thanks Professor Osamu Fujino for informing him of \cite{fujino-van}. 
He thanks Professor Yoshinori Gongyo for answering questions. 
He thanks Professor Kento Fujita for comments and warm encouragement. 
He thanks the referee for useful comments.

\section{Definitions}\label{sec2}

Throughout this note, a {\em scheme} means a separated scheme of finite type over ${\rm Spec} (\mathbb{C})$.

\begin{defn}[Demi-normal scheme, demi-normal pair, \protect{\cite[Chapter 5]{kollar-mmp}}]\label{defn--demi-normal}
A {\em demi-normal scheme} $X$ is a reduced equidimensional scheme such that $X$ satisfies the Serre's $\mathrm{S}_{2}$ condition and every codimension one point of $X$ is smooth or normal crossing singularity. 
Let $\Delta$ be an $\mathbb{R}$-divisor on $X$ whose support does not contain any codimension one singular point. 
The pair of $X$ and $\Delta$, denoted by $(X,\Delta)$, is called a {\em demi-normal pair} if $\Delta$ is effective. 
\end{defn}

\begin{defn}[Simple normal crossing singularity, {\cite[Definition 1.10]{kollar-mmp}}]\label{defn--snc}
Let $X$ be a reduced equidimensional scheme and let $\Delta$ be an $\mathbb{R}$-divisor on $X$. 
We say that the pair of $X$ and $\Delta$, denoted by $(X,\Delta)$ if there is no risk of confusion, has {\em only  simple normal crossing singularities at a point $a\in X$} if $X$ has a Zariski open subset $U \ni a$ that can be embedded in a smooth variety $M$, where $M$ has regular local coordinates $x_{1},\ldots, x_{p},\,y_{1},\ldots ,\,y_{r}$ at $a=0$ such that $U=(x_{1}\cdots x_{p}=0)$ and $\restriction{\Delta}{U}=\restriction{\sum_{i=1}^{r} \alpha_{i}(y_{i}=0)}{U}$ for some $\alpha_{i} \in \mathbb{R}$. 
We say that $(X,\Delta)$ has {\em only simple normal crossing singularities} if $(X,\Delta)$ has only simple normal crossing singularities at every point $a\in X$. 
\end{defn}

Simple normal crossing singularities are called semi-snc in \cite{kollar-mmp}. 

\begin{defn}[Semi-log canonical pair]\label{defn--slc}
Let $(X,\Delta)$ be a demi-normal pair such that $K_{X}+\Delta$ is $\mathbb{R}$-Cartier. 
In this note, a {\em prime divisor $E$ over $X$} is a prime divisor on $Y$ with a proper birational morphism $Y\to X^{\nu}$ to the normalization $X^{\nu}$ of $X$. 
As in the case of log pairs, we can define the {\em discrepancy $a(E,X,\Delta)$ of $E$ with respect to $(X,\Delta)$} for every prime divisor $E$ over $X$. 
A demi-normal pair $(X,\Delta)$ is a {\em semi-log canonical} ({\em slc}, for short) {\em pair} if $K_{X}+\Delta$ is $\mathbb{R}$-Cartier and $a(E,X,\Delta) \geq -1$ for every prime divisor $E$ over $X$.  
\end{defn}

As we defined above, we follow the definitions of demi-normal schemes, slc pairs and sdlt pairs (see Definition \ref{defn--sdlt}) in \cite[Section 5]{kollar-mmp}. 

\begin{rem}\label{rem--sdlt-differ}
In \cite{fujino1} and \cite{fujino-gongyo}, sdlt pairs are defined to be slc pairs $(X,\Delta)$ such that 
\begin{itemize}
\item[---]
every irreducible component of $X$ is normal, and 
\item[---]
the normalization of $(X,\Delta)$ is a disjoint union of dlt pairs.
\end{itemize}
On the other hand, our definition of sdlt pairs is due to \cite[Definition 5.19]{kollar-mmp}. 
The sdlt pairs in the sense of Definition \ref{defn--sdlt} are also sdlt pairs in the above sense but the converse is not true. 
Indeed, for every sdlt pair $(X,\Delta)$ in the sense of Definition \ref{defn--sdlt}, the normalization of $(X,\Delta)$ is a disjoint union of dlt pairs, and \cite[Proposition 5.20]{kollar-mmp} implies that every irreducible component of $X$ is normal.
Thus, $(X,\Delta)$ satisfies the above two conditions. 
On the other hand, as in \cite[Definition 5.19]{kollar-mmp}, the pair of $X:=(xy=zt=0)\subset \mathbb{A}^{4}_{xyzt}$ and $\Delta:=0$ is not sdlt in the sense of Definition \ref{defn--sdlt}, but $(X,0)$ is an slc pair satisfying the two conditions stated above. 
\end{rem}

Let $X$ be a demi-normal scheme. 
If every irreducible component of $X$ is normal in codimension one, then there exists a closed subscheme $Z\subset X$ of codimension at least two such that $(X\setminus Z,0)$ has only simple normal crossing singularities. 

\begin{defn}[Semi-terminal pair, \protect{\cite[Definition 2.3 (2)]{fujita}}]\label{defn--semi-ter}
Let $(X,\Delta)$ be an slc pair. 
Let $\nu\colon X^{\nu} \to X$ be the normalization, and we define $\Delta^{\nu}$ on $X^{\nu}$ by $K_{X^{\nu}}+\Delta^{\nu}=\nu^{*}(K_{X}+\Delta)$. 
Let $\lfloor \Delta^{\nu} \rfloor$ be the reduced part of $\Delta^{\nu}$. 
Then $(X,\Delta)$ is {\em semi-terminal} if $a(P,X,\Delta)\geq 0$ for all exceptional prime divisors $P$ over $X$ and the equality holds only if the center of $P$ on $X^{\nu}$ has codimension two and it is contained in $\lfloor \Delta^{\nu} \rfloor$. 
\end{defn}

\begin{defn}[An MMP step for slc pairs, \protect{\cite[Definition 11]{ambrokollar}}]\label{defn--mmpslc}
Let $\pi \colon X\to S$ be a projective morphism between quasi-projective schemes, and let $(X,\Delta)$ be an slc pair such that $\Delta$ is a $\mathbb{Q}$-divisor. 
Then, an {\em MMP step} over $S$ for $(X,\Delta)$ is the following diagram  
$$
\xymatrix%@R=14pt
{
(X,\Delta)\ar[dr]_{f}\ar@{-->}[rr]^{\phi}&&(X',\Delta')\ar[dl]^{f'}\\
&Z
}
$$
over $S$ such that
\begin{enumerate}
\item \label{mmpslc-1}
$(X',\Delta')$ is an slc pair,
\item \label{mmpslc-2}
$\phi$ is birational and it is isomorphic on an open dense subset containing all generic points of codimension one singular locus of $X$,  
\item \label{mmpslc-3}
$\Delta'=\phi_{*}\Delta$,  
\item \label{mmpslc-4}
$Z$ and $X'$ are quasi-projective schemes, the morphisms $Z \to S$ and $X'\to S$ are projective, $f$ and $f'$ are generically finite, and $f'$ has no exceptional divisors, and 
\item \label{mmpslc-5}
$-(K_{X}+\Delta)$ and $K_{X'}+\Delta'$ are ample over $Z$. 
\end{enumerate}
\end{defn}

\section{Log MMP for semi-log canonical pairs}\label{sec3}

In this section, we discuss the construction and properties of log MMP for slc pairs. 
See \cite{ambrokollar} and \cite[Section 4]{has-finite} for related results. 

\begin{lem}[\emph{cf.}~\protect{\cite{ambrokollar}}, \protect{\cite[Lemma 4.4]{has-finite}}]\label{lem--mmpstepslc}
Let $\pi \colon X\to S$ be a projective morphism between quasi-projective schemes, and let $(X,\Delta)$ be an slc pair such that $\Delta$ is a $\mathbb{Q}$-divisor. 
Let $\nu \colon (\bar{X},\bar{\Delta})\to (X,\Delta)$ be the normalization, where $K_{\bar{X}}+\bar{\Delta}=\nu^{*}(K_{X}+\Delta)$. 
Suppose that every irreducible component $(\bar{X}^{(j)},\bar{\Delta}^{(j)})$ of $(\bar{X},\bar{\Delta})$ satisfies the following condition.  
\begin{itemize}
\item[]
The relative stable base locus of $K_{\bar{X}^{(j)}}+\bar{\Delta}^{(j)}$ over $S$ does not contain the image of any prime divisor $P$ over $\bar{X}^{(j)}$ satisfying $a(P,\bar{X}^{(j)},\bar{\Delta}^{(j)})<0$. 
\end{itemize}
Let $H$ be an effective $\mathbb{Q}$-Cartier divisor such that $(X,\Delta+cH)$ is an slc pair and $K_{X}+\Delta+c H$ is nef over $S$ for some $c>0$. 
Set $\lambda={\rm inf}\left\{\mu \in \mathbb{R}_{\geq 0}\mid K_{X}+\Delta+\mu H\ \text{is nef over}\ S\right\}$.

If $\lambda>0$, then we may construct an MMP step over $S$ for $(X,\Delta)$ as in Definition \ref{defn--mmpslc}
$$
\xymatrix%@R=14pt
{
(X,\Delta)\ar[dr]_{f}\ar@{-->}[rr]^{\phi}&&(X',\Delta')\ar[dl]^{f'}\\
&Z
}
$$
such that $f\colon X\to Z$ is a contraction of a $(K_{X}+\Delta)$-negative extremal ray $R$ over $S$ $($see \cite[Theorem 1.19]{fujino-fund-slc}$)$ satisfying $(K_{X}+\Delta+\lambda H)\cdot R=0$. 
In particular, $(X',\Delta'+\lambda H')$ is an slc pair and $K_{X'}+\Delta'+\lambda H'$ is nef over $S$, where $H'=\phi_{*}H$. 
Furthermore, for the normalization $\nu'\colon(\bar{X}',\bar{\Delta}')\to (X',\Delta')$, where $K_{\bar{X}'}+\bar{\Delta}'=\nu'^{*}(K_{X}+\Delta)$, every irreducible component $(\bar{X}'^{(j)},\bar{\Delta}'^{(j)})$ of $(\bar{X}',\bar{\Delta}')$ satisfies the following condition.
\begin{itemize}
\item[]
The relative stable base locus of $K_{\bar{X}'^{(j)}}+\bar{\Delta}'^{(j)}$ over $S$ does not contain the image of any prime divisor $P'$ over $\bar{X}'^{(j)}$ satisfying $a(P',\bar{X}'^{(j)},\bar{\Delta}'^{(j)})<0$. 
\end{itemize}
\end{lem}

\begin{proof}
The proof is essentially the same as \cite[Subsection 4.2]{fujita}. By \cite[Theorem 1.19 (1)]{fujino-fund-slc}, we can find a $(K_{X}+\Delta)$-negative extremal ray $R$ such that $(K_{X}+\Delta+\lambda H)\cdot R=0$. 
Since $K_{X}+\Delta$ and $H$ are $\mathbb{Q}$-Cartier, $\lambda$ is a rational number. 
From now on, we follow \cite[Proof of Lemma 4.4]{has-finite}. 

By \cite[Theorem 1.19 (3)]{fujino-fund-slc}, we get a projective contraction $f\colon X\to Z$ over $S$ that exactly contracts the extremal ray $R$. 
We put $\bar{H}=\nu^{*}H$. 
By our assumption on $(\bar{X}^{(j)},\bar{\Delta}^{(j)})$, we see that $K_{\bar{X}^{(j)}}+\bar{\Delta}^{(j)}$ is $\mathbb{Q}$-linearly equivalent to an effective $\mathbb{Q}$-Cartier divisor over $S$. 
On the other hand, by construction $-(K_{\bar{X}^{(j)}}+\bar{\Delta}^{(j)})$ is ample over $Z$. 
Thus, the morphism $\bar{X}^{(j)} \to Z$ is generically finite for every $j$, so $f\colon X\to Z$ is birational. 
By \cite[Theorem 1.1]{birkar-flip}, the lc pair $(\bar{X}^{(j)},\bar{\Delta}^{(j)})$ has the log canonical model $\bar{\phi}^{(j)} \colon (\bar{X}^{(j)},\bar{\Delta}^{(j)})\dashrightarrow (\bar{X}'^{(j)},\bar{\Delta}'^{(j)})$ over $Z$ as in \cite[Definition 3.50]{kollar-mori}. 
Put $(\bar{X}',\bar{\Delta}')={\coprod_{j}}(\bar{X}'^{(j)},\bar{\Delta}'^{(j)})$. 
We consider the diagram
$$
\xymatrix%@R=14pt
{
(\bar{X},\bar{\Delta})\ar[dr]\ar@{-->}[rr]&&(\bar{X}',\bar{\Delta}')\ar[dl]\\
&Z
}
$$
over $S$, where the birational map $\bar{X}\dashrightarrow \bar{X}'$ is defined by $\amalg_{j}\bar{\phi}^{(j)}\colon \amalg_{j}\bar{X}^{(j)} \dashrightarrow \amalg_{j}\bar{X}'^{(j)}$. 
By the same argument as in \cite[Proof of Lemma 4.4]{has-finite}, the diagram is an MMP step over $S$ for $(\bar{X},\bar{\Delta})$ as in Definition \ref{defn--mmpslc}. 
Moreover, the non-isomorphic locus of  $\bar{\phi}^{(j)}$ is contained in the relative stable base locus of $K_{\bar{X}^{(j)}}+\bar{\Delta}^{(j)}$, which does not contain the image of any prime divisor $P$ over $\bar{X}^{(j)}$ satisfying $a(P,\bar{X}^{(j)},\bar{\Delta}^{(j)})<0$. 
Then \cite[Proof of Theorem 9]{ambrokollar} works. 
In this way, we get an slc pair $(X',\Delta')$ and a projective morphism $f'\colon X'\to Z$ such that $(\bar{X}',\bar{\Delta}')$ is the normalization of $(X',\Delta')$, and the diagram
$$
\xymatrix%@R=14pt
{
(X,\Delta)\ar[dr]_{f}\ar@{-->}[rr]^{\phi}&&(X',\Delta')\ar[dl]^{f'}\\
&Z
}
$$
 is an MMP step over $S$ for $(X,\Delta)$ as in Definition \ref{defn--mmpslc}. 
As in \cite[Proof of Lemma 4.4]{has-finite}, the relative stable base locus of $K_{\bar{X}'^{(j)}}+\bar{\Delta}'^{(j)}$ over $S$ does not contain the image of any prime divisor $P'$ over $\bar{X}'^{(j)}$ satisfying $a(P',\bar{X}'^{(j)},\bar{\Delta}'^{(j)})<0$. 
We pick a positive integer $m$ such that $m(K_{X}+\Delta+\lambda H)$ is Cartier. 
By \cite[Theorem~1.19~(4)]{fujino-fund-slc} and using the fact that $(K_{X}+\Delta+\lambda H)\cdot R=0$, there exists a line bundle $\mathcal{L}$ on $Z$ such that $\mathcal{O}_{X}(m(K_{X}+\Delta+\lambda H)) \simeq f^{*}\mathcal{L}$. 
Thus, $K_{X'}+\Delta'+\lambda H'$ is nef over $S$, where $H'=\phi_{*}H$, and the slc property of $(X',\Delta'+\lambda H')$ follows from that the normalization of $(X',\Delta'+\lambda H')$ is a disjoint union of lc pairs. 
\end{proof}

\begin{rem}\label{rem--mmpstepslc}
By \cite[Theorem 1.19 (4)]{fujino-fund-slc}, the MMP step constructed in Lemma \ref{lem--mmpstepslc} satisfies the property that for every $\mathbb{Q}$-Cartier divisor $D$ on $X$ the birational transform $\phi_{*}D$ is also $\mathbb{Q}$-Cartier. 
Especially, the MMP step in Lemma \ref{lem--mmpstepslc} preserves the $\mathbb{Q}$-Gorenstein property. 
Moreover, the above proof (see also \cite[Proof of Theorem 9]{ambrokollar}) shows that the MMP step constructed in Lemma \ref{lem--mmpstepslc} satisfies that the inverse map $\phi^{-1}\colon X'\dashrightarrow X$ is an isomorphism on an open dense subset containing all the generic points of codimension one singular locus of $X'$.    
As we will see in Lemma \ref{lem--mmp-preserve-sdlt} below, the property of $\phi^{-1}$ implies that the MMP step as in Lemma \ref{lem--mmpstepslc} preserves the sdlt property.
\end{rem}

\begin{defn}[MMP for slc pairs with scaling, \emph{cf.}~\protect{\cite[Subsection 4.2]{fujita}}]\label{defn--sequencemmp}
We explain the construction of a sequence of MMP steps for slc pairs. 

Let $\pi \colon X\to S$ be a projective morphism between quasi-projective schemes, and let $(X,\Delta)$ an slc pair such that $\Delta$ is a $\mathbb{Q}$-divisor. 
Let $\nu \colon (\bar{X},\bar{\Delta})\to (X,\Delta)$ be the normalization, where $K_{\bar{X}}+\bar{\Delta}=\nu^{*}(K_{X}+\Delta)$. 
Suppose that every irreducible component $(\bar{X}^{(j)},\bar{\Delta}^{(j)})$ of $(\bar{X},\bar{\Delta})$ satisfies the following condition. 
\begin{itemize} \item[] The relative stable base locus of $K_{\bar{X}^{(j)}}+\bar{\Delta}^{(j)}$ over $S$ does not contain the image of any prime divisor $P$ over $\bar{X}^{(j)}$ satisfying $a(P,\bar{X}^{(j)},\bar{\Delta}^{(j)})<0$.  \end{itemize}
Let $H$ be an effective $\mathbb{Q}$-Cartier divisor on $X$ such that $(X,\Delta+cH)$ is an slc pair and $K_{X}+\Delta+c H$ is nef over $S$ for some $c>0$. 

We define $\lambda:={\rm inf}\left\{\mu \in \mathbb{R}_{\geq 0}\mid K_{X}+\Delta+\mu H \ \text{is nef over}\ S\right\}.$ 
If $\lambda=0$, there is nothing to do. 
If $\lambda>0$, by Lemma \ref{lem--mmpstepslc} we can construct an MMP step over $S$
$$\phi \colon (X,\Delta)\dashrightarrow (X_{1},\Delta_{1})$$
such that $(X_{1},\Delta_{1}+\lambda H_{1})$ is an slc pair and $K_{X_{1}}+\Delta_{1}+\lambda H_{1}$ is nef over $S$, where $H_{1}=\phi_{*}H$. 
We define $\lambda_{1}:={\rm inf}\left\{\mu \in \mathbb{R}_{\geq 0}\mid K_{X_{1}}+\Delta_{1}+\mu H_{1}\ \text{is nef over}\ S\right\}$. 
Then $\lambda_{1}\leq \lambda$. 
If $\lambda_{1}=0$, then we stop the discussion. 
If $\lambda_{1}>0$, let $\nu_{1} \colon (\bar{X}_{1},\bar{\Delta}_{1})\to (X_{1},\Delta_{1})$ be the normalization, where $K_{\bar{X}_{1}}+\bar{\Delta}_{1}=\nu_{1}^{*}(K_{X}+\Delta)$. 
Then, every irreducible component $(\bar{X}_{1}^{(j)},\bar{\Delta}_{1}^{(j)})$ of $(\bar{X}_{1},\bar{\Delta}_{1})$ satisfies the condition of Lemma \ref{lem--mmpstepslc}. 
Hence, we may apply Lemma \ref{lem--mmpstepslc} to $(X_{1},\Delta_{1})\to S$ and $H_{1}$. 
By Lemma \ref{lem--mmpstepslc}, we get an MMP step over $S$
$$(X_{1},\Delta_{1})\dashrightarrow (X_{2},\Delta_{2}).$$
By repeating this discussion, we get a sequence of MMP steps over $S$
$$(X,\Delta)=:(X_{0},\Delta_{0}) \dashrightarrow (X_{1},\Delta_{1}) \dashrightarrow \cdots \dashrightarrow (X_{i},\Delta_{i})\dashrightarrow \cdots$$
such that  if we put $\lambda_{i}={\rm inf}\left\{\mu \in \mathbb{R}_{\geq 0}\mid K_{X_{i}}+\Delta_{i}+\mu H_{i}\ \text{is nef over}\ S\right\}$ for each $i \geq 0$, where $H_{i}$ is the birational transform of $H$ on $X_{i}$, then $\lambda_{i} \geq \lambda_{i+1}$ for all $i\geq 0$. 
\end{defn}

\begin{lem}\label{lem--mmp-preserve-sdlt}
Let $\pi \colon X\to S$ be a projective morphism between quasi-projective schemes, and let $(X,\Delta)$ be an slc pair such that $\Delta$ is a $\mathbb{Q}$-divisor. 
Let $$
\xymatrix%@R=14pt
{
(X,\Delta)\ar[dr]_{f}\ar@{-->}[rr]^{\phi}&&(X',\Delta')\ar[dl]^{f'}\\
&Z
}
$$
be an MMP step over $S$ as in Definition \ref{defn--mmpslc} such that the inverse map $\phi^{-1}\colon X'\dashrightarrow X$ is an isomorphism on an open dense subset containing all the generic points of codimension one singular locus of $X'$. 
Suppose further that $(X,\Delta)$ is sdlt. 
Then $(X',\Delta')$ is sdlt.
\end{lem}

\begin{proof}
Let $\nu' \colon \bar{X}'\to X'$ be the normalization, and we define $\bar{\Delta}'$ by $K_{\bar{X}'}+\bar{\Delta}'=\nu'^{*}(K_{X}+\Delta)$. 
Similarly, let $\nu \colon \bar{X} \to X$ be the normalization, and we define $\bar{\Delta}$ by $K_{\bar{X}}+\bar{\Delta}=\nu^{*}(K_{X}+\Delta)$. 
Since $(X,\Delta)$ is an sdlt pair, $(\bar{X},\bar{\Delta})$ is a disjoint union of dlt pairs. 
By \cite[Lemma 12]{ambrokollar}, the birational map $\bar{\phi}\colon (\bar{X},\bar{\Delta}) \dashrightarrow (\bar{X}',\bar{\Delta}')$ is an MMP step over $S$ for $(\bar{X},\bar{\Delta})$. 
Thus, $(\bar{X}',\bar{\Delta}')$ is a disjoint union of dlt pairs. 

Let $P$ be a prime divisor over $X'$ such that $a(P,X',\Delta')=-1$, let $V'$ be the center of $P$ on $X'$, and let $\eta_{V'}$ is the generic point of $V'$. 
To prove that $(X',\Delta')$ is sdlt, we need to prove that $(X',\Delta')$ has only simple normal crossing singularities at $\eta_{V'}$. 
Since $(X,\Delta)$ is sdlt, it is sufficient to prove that $\phi^{-1}\colon X'\dashrightarrow X$ is an isomorphism on a neighborhood of $\eta_{V'}$. 
The latter condition follows if $\bar{\phi}^{-1}\colon \bar{X}' \dashrightarrow \bar{X}$ is an isomorphism on a neighborhood of $\nu'^{-1}(\eta_{V'})$. 
Here, we used the fact that $\phi$ and $\phi^{-1}$ are isomorphisms on open subsets containing all the generic points of codimension one singular loci. 
Since $\bar{\phi}\colon (\bar{X},\bar{\Delta}) \dashrightarrow (\bar{X}',\bar{\Delta}')$ is an MMP step for $(\bar{X},\bar{\Delta})$, the argument as in \cite[Lemma 3.38]{kollar-mori} shows that $\bar{\phi}^{-1}$ is an isomorphism on an open subset containing $\nu'^{-1}(\eta_{V'})$ if every irreducible component of 
 $\overline{\nu'^{-1}(\eta_{V'})}$ is an lc center of $(\bar{X}',\bar{\Delta}')$. 
In this way, to prove that $(X',\Delta')$ is sdlt, it is sufficient to prove that every irreducible component of $\overline{\nu'^{-1}(\eta_{V'})}$ is an lc center of $(\bar{X}',\bar{\Delta}')$. 

We note that $\nu'^{-1}(\eta_{V'})$ consists of finite points. 
We define $p \in \nu'^{-1}(\eta_{V'})$ to be the point such that $\overline{\{p\}}$ is the center of $P$ on $\bar{X}'$. 
If $p = \nu'^{-1}(\eta_{V'})$, then there is nothing to prove, so we may assume that $\nu'^{-1}(\eta_{V'})$ has at least two points. 
Let $D'\subset \bar{X}'$ be the conductor of $X'$, and let $D'^{\nu}=\amalg_{i}D'_{i}$ denote the normalization of $D'$. 
Then $\nu'(D')\ni \eta_{V'}$, and $D'^{\nu}=\amalg_{i}D'_{i}$ is the irreducible decomposition of $D'$ because $(\bar{X}',\bar{\Delta}')$ is a disjoint union of dlt pairs. 
Let $\Delta_{D'_{i}}$ be a $\mathbb{Q}$-divisor on $D'_{i}$ defined by adjunction $K_{D'_{i}}+\Delta_{D'_{i}}=\restriction{(K_{\bar{X}'}+\bar{\Delta}')}{D'_{i}}$. 

We pick any point $q \in \nu'^{-1}(\eta_{V'})$, and we will prove that $\overline{\{q\}}$ is an lc center of $(\bar{X}',\bar{\Delta}')$. 
The construction of $X'$ with the gluing theory as in \cite[Corollary 5.33]{kollar-mmp} shows that $X'$ is the geometric quotient of $\bar{X}'$ by an involution $\tau$ of $\amalg_{i}D'_{i}$ such that $\tau_{*}(\amalg_{i}\Delta_{D'_{i}})=\amalg_{i}\Delta_{D'_{i}}$. 
From this (see also \cite[Definition 5.31]{kollar-mmp} and \cite[9.3]{kollar-mmp}), we can find $p_{1},\cdots,\,p_{m}\in \nu'^{-1}(\eta_{V'})$ such that the following conditions hold. 
\begin{itemize}
\item[---]
$p_{1}=p$ and $p_{m}=q$, and
\item[---]
for each $1 \leq l < m$, there is an index $i(l)$ such that $p_{l} \in D'_{i(l)}$ and the morphism $\restriction{\tau}{D'_{i(l)}}\colon D'_{i(l)} \to \tau\left(D'_{i(l)}\right)$ sends $p_{l}$ to $p_{l+1}$ 
(note that we can have $D_{i(l)}=D_{i(l')}$ for some $l \neq l'$, and $\tau(D'_{i(l)})$ is not necessarily equal to $D'_{i(l+1)}$).
\end{itemize}
We will show that $\overline{\{p_{l}\}}$ is an lc center of $(\bar{X}',\bar{\Delta}')$ for every $1\leq l \leq m$. 
The case of $l=1$ is clear from the definition of $p$. 
Suppose further that $\overline{\{p_{l}\}}$ is an lc center of $(\bar{X}',\bar{\Delta}')$. 
Note that every $D'_{i}$ is an lc center of $(\bar{X}',\bar{\Delta}')$, and $(\bar{X}',\bar{\Delta}')$ is a disjoint union of dlt pairs. 
So $\overline{\{p_{l}\}}$ is an lc center of $(D'_{i(l)},\Delta_{D'_{i(l)}})$. 
There is an index $k$ such that $D'_{k}=\tau\left(D'_{i(l)}\right)$ and $\Delta_{D'_{k}}=\tau_{*}\Delta_{D'_{i(l)}}$. 
Then $\overline{\{p_{l+1}\}}$ is an lc center of $\left(D'_{k},\Delta_{D'_{k}}\right)$ because $\restriction{\tau}{D'_{i(l)}}$ is an isomorphsim. 
Then $\overline{\{p_{l+1}\}}$ is an lc center of $(\bar{X}',\bar{\Delta}')$ since $(\bar{X}',\bar{\Delta}')$ is a disjoint union of dlt pairs and $D'_{k}$ is an lc center of $(\bar{X}',\bar{\Delta}')$ (see, for example, \cite[Theorem 4.16]{kollar-mmp}). 
By induction on $l$, we see that $\overline{\{p_{l}\}}$ is an lc center of $(\bar{X}',\bar{\Delta}')$ for all $1\leq l \leq m$. 
So $\overline{\{q\}}=\overline{\{p_{m}\}}$ is an lc center of $(\bar{X}',\bar{\Delta}')$. 

We have proved that every irreducible component of the closure of $\nu'^{-1}(\eta_{V'})$ is an lc center of $(\bar{X}',\bar{\Delta}')$. 
Thus, the birational map $\bar{\phi}^{-1}\colon \bar{X}' \dashrightarrow \bar{X}$ is an isomorphism on an open subset containing $\nu'^{-1}(\eta_{V'})$. 
Then the birational map $\phi^{-1}\colon X'\dashrightarrow X$ is an isomorphism on a neighborhood of the generic point of $V'$. 
Thus, $(X',\Delta')$ is sdlt.  
\end{proof}

\section{Proof of main result}\label{sec4}

\begin{lem}[\emph{cf.}~\protect{\cite[Theorem 1.2]{fujino-fund-slc}, \cite[Remark 1.5]{fujino-fund-slc}}]\label{lem--resol}
Let $(X,\Delta)$ be a quasi-projective slc pair such that every irreducible component of $X$ is normal in codimension one. 
Then there is a projective birational morphism $\phi\colon \tilde{X}\to X$ such that if we define an $\mathbb{R}$-divisor $\Xi$ on $\tilde{X}$ by $K_{\tilde{X}}+\Xi=\phi^{*}(K_{X}+\Delta)$, then the following properties hold.
\begin{enumerate}[label=(\roman*)]
\item \label{lem--(i)}
$(\tilde{X},\Xi)$ is a globally embedded simple normal crossing pair, in other words, there exists a smooth quasi-projective variety $M$ of dimension ${\rm dim}X+1$ and an snc $\mathbb{R}$-divisor $B$ on $M$ such that 
\begin{itemize}
\item[---]
$\tilde{X}$ is an snc divisor on $M$ and $\tilde{X}+B$ is an snc $\mathbb{R}$-divisor on $M$,  
\item[---]
$\tilde{X}$ and $B$ have no common components, and 
\item[---]
$\Xi=\restriction{B}{\tilde{X}}$,
\end{itemize}
\item \label{lem--(ii)}
$\phi$ is an isomorphism over all the generic points of codimension one singular locus of $X$, 
\item \label{lem--(iii)}
if we write $\Xi=\tilde{\Delta}-\tilde{E}$ where $\tilde{\Delta}$ and $\tilde{E}$ have no common components, then $\tilde{\Delta}$ and $\tilde{E}$ are both $\mathbb{R}$-Cartier,  
\item  \label{lem--(iv)}
for every point $x\in \tilde{X}$, the number of irreducible components of $\tilde{X}$ containing $x$ is at most two, and 
\item  \label{lem--(v)}
every codimension one singular point of $\tilde{X}$ maps to a codimension one singular point of $X$. 
\end{enumerate} 
\end{lem}

\begin{proof}
By \cite[Proof of Theorem 1.2]{fujino-fund-slc} and \cite[Remarks 4.1--4.3]{fujino-fund-slc} with the aid of \cite{bm}, \cite{bvp}, and \cite[Corollary 10.55]{kollar-mmp}, we may construct a projective birational morphism $\phi\colon \tilde{X}\to X$ such that defining an $\mathbb{R}$-divisor $\Xi$ on $\tilde{X}$ by $K_{\tilde{X}}+\Xi=\phi^{*}(K_{X}+\Delta)$, then $(\tilde{X},\Xi)$ satisfies \ref{lem--(i)} of Lemma \ref{lem--resol}.
Then, \ref{lem--(ii)} of Lemma \ref{lem--resol} directly follows from the construction of $\phi$ in \cite[Proof of Theorem 1.2]{fujino-fund-slc} (see also \cite{bm}, \cite{bvp}, \cite[Corollary 10.55]{kollar-mmp}), or we can check \ref{lem--(ii)} of Lemma \ref{lem--resol} as follows: 
by \cite[Theorem 1.2 (4)]{fujino-fund-slc}, we see that $\phi_{*}\mathcal{O}_{\tilde{X}}\simeq \mathcal{O}_{X}$.
In particular, $\phi$ has connected fibers over all codimension one singular points. 
Since $\phi$ is birational, \ref{lem--(ii)} of Lemma \ref{lem--resol} holds. 
With \ref{lem--(i)} of Lemma \ref{lem--resol}, we show \ref{lem--(iii)} of Lemma \ref{lem--resol}. 
Let $B_{+}$ (resp.~$B_{-}$) be the positive (resp.~negative) part of $B$ in \ref{lem--(i)}. 
Then $\Xi=\restriction{B_{+}}{\tilde{X}}-\restriction{B_{-}}{\tilde{X}}$. 
Since $\tilde{X}+B$ is an snc $\mathbb{R}$-divisor on $M$ and $\tilde{X}$, $B_{+}$, and $B_{-}$ have no common components each other, the codimension of ${\rm Supp}\restriction{B_{+}}{\tilde{X}} \cap {\rm Supp}\restriction{B_{-}}{\tilde{X}}$ in $\tilde{X}$ is at least two. 
This shows that $\restriction{B_{+}}{\tilde{X}}$ and $\restriction{B_{-}}{\tilde{X}}$ have no common components. 
By definitions of $\tilde{\Delta}$ and $\tilde{E}$, the relations $\tilde{\Delta}=\restriction{B_{+}}{\tilde{X}}$ and $\tilde{E}=\restriction{B_{-}}{\tilde{X}}$ hold. 
Therefore, \ref{lem--(iii)} of Lemma \ref{lem--resol} holds. 
If $\tilde{X}$ does not satisfy \ref{lem--(iv)} and \ref{lem--(v)} of Lemma \ref{lem--resol}, then  we take a blow-up of a stratum of $(M,\tilde{X})$ of codimension greater than or equal to two, and we replace $\tilde{X}$ by the birational transform. 
Note that the replacement keeps \ref{lem--(i)}--\ref{lem--(iii)}. 
By repeating the replacement, we get $\tilde{X}$ satisfying \ref{lem--(iv)} and \ref{lem--(v)} of Lemma \ref{lem--resol}. 
\end{proof}

\begin{proof}[Proof of Theorem \ref{thm--crepant-sdlt}]
Let $(X,\Delta)$ be as in Theorem \ref{thm--crepant-sdlt}. 
By using Lemma \ref{lem--resol}, we get a projective birational morphism $\phi\colon \tilde{X}\to X$ satisfying the conditions \ref{lem--(i)}--\ref{lem--(v)} of Lemma \ref{lem--resol}. 
We freely use the notation as in \ref{lem--(i)}--\ref{lem--(v)} of Lemma \ref{lem--resol}. 
We may write
$$K_{\tilde{X}}+\tilde{\Delta}=\phi^{*}(K_{X}+\Delta)+\tilde{E}$$
where $\tilde{\Delta}$ and $\tilde{E}$ are effective $\mathbb{Q}$-divisors on $\tilde{X}$ which have no common components. 
As in the proof of Lemma \ref{lem--resol}, putting $B_{+}$ as the effective part of $B$ then $\tilde{\Delta}=\restriction{B_{+}}{\tilde{X}}$. 
Therefore, $(\tilde{X},\tilde{\Delta})$ has only simple normal crossing singularities. 

We show that $(\tilde{X},\tilde{\Delta})$ is an slc pair and the morphism $(\tilde{X},\tilde{\Delta})\to X$ satisfies the condition of Lemma \ref{lem--mmpstepslc}. 
By construction, $(\tilde{X},\tilde{\Delta})$ is an slc pair such that $\tilde{\Delta}$ is a $\mathbb{Q}$-divisor. 
Let $\tilde{X}^{\nu}\to \tilde{X}$ be the normalization, and let $\tilde{\Delta}^{\nu}$ be a $\mathbb{Q}$-divisor on $\tilde{X}^{\nu}$ such that $K_{\tilde{X}^{\nu}}+\tilde{\Delta}^{\nu}$ is equal to the pullback of $K_{\tilde{X}}+\tilde{\Delta}$. 
Pick an irreducible component $\tilde{X}^{(j)}$ of $\tilde{X}^{\nu}$, and we put $\tilde{\Delta}^{(j)}$ and $\tilde{E}^{(j)}$ as the restrictions of $\tilde{\Delta}^{\nu}$ and $\tilde{E}$ to $\tilde{X}^{(j)}$, respectively. 
Then the relation
$$K_{\tilde{X}^{(j)}}+\tilde{\Delta}^{(j)}\sim_{\mathbb{Q},X}\tilde{E}^{(j)}$$
holds. 
By \ref{lem--(i)} of Lemma \ref{lem--resol} and the definitions of $\tilde{\Delta}$ and $\tilde{E}$, the pair $(\tilde{X}^{(j)}, \tilde{\Delta}^{(j)}+\tilde{E}^{(j)})$ is log smooth and $\tilde{\Delta}^{(j)}$ and $\tilde{E}^{(j)}$ have no common components. 
By simple computations of discrepancies with \cite[Lemma 2.29]{kollar-mori} and \cite[Lemma 2.45]{kollar-mori}, we can check that ${\rm Supp}\tilde{E}^{(j)}$ does not contain the image of any prime divisor $P$ over $\bar{X}^{(j)}$ satisfying $a(P,\bar{X}^{(j)},\bar{\Delta}^{(j)})<0$. 
In particular, $(\tilde{X},\tilde{\Delta})$ satisfies the condition of Lemma \ref{lem--mmpstepslc}. 
Therefore, we see that $(\tilde{X},\tilde{\Delta})$ is an slc pair and the morphism $(\tilde{X},\tilde{\Delta})\to X$ satisfies the condition of Lemma \ref{lem--mmpstepslc}. 

Let $\tilde{H}$ be a general ample $\mathbb{Q}$-Cartier divisor on $\tilde{X}$ such that $K_{\tilde{X}}+\tilde{\Delta}+\tilde{H}$ is ample over $X$. 
By the above argument, we may apply Lemma \ref{lem--mmpstepslc} to the morphism $(\tilde{X},\tilde{\Delta})\to X$. 
By Definition \ref{defn--sequencemmp}, we get a sequence of MMP steps over $X$ 
$$(\tilde{X},\tilde{\Delta})=:(\tilde{X}_{0},\tilde{\Delta}_{0})\dashrightarrow  (\tilde{X}_{1},\tilde{\Delta}_{1})\dashrightarrow \cdots \dashrightarrow (\tilde{X}_{i},\tilde{\Delta}_{i})\dashrightarrow \cdots$$
such that if we put $\lambda_{i}={\rm inf}\left\{\mu \in \mathbb{R}_{\geq 0}\mid K_{\tilde{X}_{i}}+\tilde{\Delta}_{i}+\mu \tilde{H}_{i}\ \text{is nef over}\ X\right\}$, where $\tilde{H}_{i}$ is the birational transform of $\tilde{H}$ on $\tilde{X}_{i}$, then $\lambda_{i} \geq \lambda_{i+1}$ for every $i\geq 0$.  

We take the normalization $\nu_{i}\colon \tilde{X}_{i}^{\nu} \to \tilde{X}_{i}$ and we define $\tilde{\Delta}_{i}^{\nu}$ by $K_{\tilde{X}_{i}^{\nu}}+\tilde{\Delta}_{i}^{\nu}=\nu_{i}^{*}(K_{\tilde{X}_{i}}+\tilde{\Delta}_{i})$ for each $i\geq 0$. 
Let $\amalg_{j}(\tilde{X}_{i}^{(j)},\tilde{\Delta}_{i}^{(j)})$ be the irreducible decomposition of $(\tilde{X}_{i}^{\nu}, \tilde{\Delta}_{i}^{\nu})$, and let $\tilde{H}_{i}^{(j)}$ be the pullback of $\tilde{H}_{i}$ to $\tilde{X}_{i}^{(j)}$. 
By \cite[Lemma 12]{ambrokollar}, for all $j$, each step of the sequence of birational maps 
$$(\tilde{X}^{(j)},\tilde{\Delta}^{(j)})=(\tilde{X}_{0}^{(j)},\tilde{\Delta}_{0}^{(j)})\dashrightarrow  (\tilde{X}_{1}^{(j)},\tilde{\Delta}_{1}^{(j)})\dashrightarrow \cdots \dashrightarrow (\tilde{X}_{i}^{(j)},\tilde{\Delta}_{i}^{(j)})\dashrightarrow \cdots$$
is an MMP step over $X$ as in Definition \ref{defn--mmpslc}. 
Then $(\tilde{X}_{i}^{(j)},\tilde{\Delta}_{i}^{(j)}+\lambda_{i}\tilde{H}_{i}^{(j)})$ is a weak lc model of $(\tilde{X}^{(j)},\tilde{\Delta}^{(j)}+\lambda_{i}\tilde{H}^{(j)})$ over $X$ (\cite[Definition 3.50]{kollar-mori}) for all $i$. 
Since $K_{\tilde{X}^{(j)}}+\tilde{\Delta}^{(j)}\sim_{\mathbb{Q},X}\tilde{E}^{(j)}$, $(\tilde{X}^{(j)},\tilde{\Delta}^{(j)})$ is dlt, and $\tilde{H}^{(j)}$ is ample, the argument of \cite[Proof of Theorem 3.4]{birkar-flip} implies that there exists a positive integer $m$ such that for all $j$ the MMP for $(\tilde{X}^{(j)},\tilde{\Delta}^{(j)})$ terminates after $m$-th steps. 
Then $\lambda_{m}=0$, and clearly $K_{\tilde{X}_{m}^{(j)}}+\tilde{\Delta}_{m}^{(j)} \sim_{\mathbb{Q},X} 0$ for all $j$. 
So, the MMP over $X$ for $(\tilde{X},\tilde{\Delta})$ terminates with $(\tilde{X}_{m},\tilde{\Delta}_{m})$ such that $K_{\tilde{X}_{m}}+\tilde{\Delta}_{m}\sim_{\mathbb{Q},X}0$.

We put $(Y,\Gamma)=(\tilde{X}_{m},\tilde{\Delta}_{m})$, and we denote the morphism $\tilde{X}_{m}\to X$ by $f$. 
We show that $(Y,\Gamma)$ is sdlt and $f\colon Y\to X$ is the desired morphism. 

Firstly, by \ref{lem--(i)} of Lemma \ref{lem--resol}, $(\tilde{X}_{0},\tilde{\Delta}_{0})$  has only simple normal crossing singularities. 
Using Lemma \ref{lem--mmp-preserve-sdlt} repeatedly, we see that $(Y,\Gamma)$ is sdlt. 

Secondly, by \ref{lem--(ii)} of Lemma \ref{lem--resol} and construction of MMP steps, $f$ is an isomorphism over all codimension one singular points of $X$. 
By \ref{lem--(v)} of Lemma \ref{lem--resol}, we see that every codimension one singular point of $Y$ maps to a codimension one singular point of $X$. 
So we can find an open dense subset $U \subset X$ such that $U$ (resp.~$f^{-1}(U))$ contains all the codimension one singular points of $X$ (resp.~$Y$) and $f$ is an isomorphism over $U$.

Thirdly, by construction the relation $K_{Y}+\Gamma=f^{*}(K_{X}+\Delta)$ holds. 

Finally, we show that $K_{Y}$ is $\mathbb{Q}$-Cartier and $(Y,0)$ is semi-terminal. 
By \ref{lem--(i)} of Lemma \ref{lem--resol}, $K_{\tilde{X}}$ is $\mathbb{Q}$-Cartier.  Then, by using Remark \ref{rem--mmpstepslc} repeatedly, we see that $K_{Y}$ is $\mathbb{Q}$-Cartier. 
Let $\nu_{Y}\colon Y^{\nu}\to Y$ be the normalization. 
We can write $K_{Y^{\nu}}+D_{Y^{\nu}}=\nu_{Y}^{*}K_{Y}$, where $D_{Y^{\nu}}$ is the conductor. 
Let $P$ be an exceptional prime divisor over $Y$ and $c_{Y^{\nu}}(P)$ the center of $P$ on $Y^{\nu}$. 
By \ref{lem--(i)} and \ref{lem--(iv)} of Lemma \ref{lem--resol}, the pair of $\tilde{X}^{\nu}$ and the conductor $D_{\tilde{X}^{\nu}}\subset \tilde{X}^{\nu}$ is a log smooth pair, and $D_{\tilde{X}^{\nu}}$ is a disjoint union of smooth varieties. 
So, if the birational map $\tilde{X}^{\nu}\dashrightarrow Y^{\nu}$ is isomorphic at the generic point of $c_{Y^{\nu}}(P)$, then on a neighborhood of the generic point of $c_{Y^{\nu}}(P)$ the pair $(Y^{\nu},D_{Y^{\nu}})$ is log smooth and $D_{Y^{\nu}}$ is irreducible. 
From this, we see that $a(P,Y,0)=a(P,Y^{\nu},D_{Y^{\nu}}) \geq 0$ and the equality holds only if ${\rm codim}_{Y^{\nu}}(c_{Y^{\nu}}(P))=2$ and $c_{Y^{\nu}}(P) \subset D_{Y^{\nu}}$. 
If $\tilde{X}^{\nu}\dashrightarrow Y^{\nu}$ is not isomorphic at the generic point of $c_{Y^{\nu}}(P)$, then 
$$a(P,Y,0)=a(P,Y^{\nu},D_{Y^{\nu}}) \geq a(P,Y^{\nu},D_{Y^{\nu}}+\nu_{Y}^{*}\Gamma)>a(P,\tilde{X}^{\nu},\tilde{\Delta}^{\nu})$$
where the final inequality follows from the fact that $(\tilde{X}^{\nu},\tilde{\Delta}^{\nu})\dashrightarrow (Y^{\nu}, D_{Y^{\nu}}+\nu_{Y}^{*}\Gamma)$ is a sequence of MMP steps by \cite[Lemma 12]{ambrokollar}.  
Since $\tilde{X}^{\nu}\dashrightarrow Y^{\nu}$ is not isomorphic at the generic point of $c_{Y^{\nu}}(P)$, the center of $P$ on $\tilde{X}^{\nu}$ is contained in ${\rm Supp}\tilde{E}$. 
By \ref{lem--(i)} of Lemma \ref{lem--resol}, we have $a(P,\tilde{X}^{\nu},\tilde{\Delta}^{\nu})\geq 0$. 
Thus $a(P,Y,0)>0$. 
In this way, we see that $(Y,0)$ is semi-terminal. 
\end{proof}

Finally, we introduce a result of crepant sdlt model for stable morphisms.

\begin{rem}\label{thm--crepant-sdlt-stable-mor}
Let $g\colon X\to C$ be a projective flat morphism between quasi-projective schemes such that $C$ is a smooth curve. 
Let $(X,\Delta)$ be a demi-normal pair such that $K_{X}+\Delta$ is $\mathbb{Q}$-Cartier and every irreducible component of $X$ is normal in codimension one. 
Suppose that $(X,\Delta+g^{*}c)$ is an slc pair for every closed point $c\in C$. 
By the proofs of Theorem \ref{thm--crepant-sdlt} and \cite[Corollary 10.46]{kollar-mmp} with the aid of \cite{bvp} and \cite{fujino-fund-slc},
we may find finitely many closed points $c_{1},\ldots, c_{m}\in C$ and 
an sdlt pair $(Y,\Gamma)$ with a projective birational morphism $f\colon Y\to X$ satisfying the following properties.
\begin{itemize}
\item[---]
There is an open dense subset $U \subset X$ such that $U$ (resp.~$f^{-1}(U))$ contains all the codimension one singular points of $X$ (resp.~$Y$) and $f$ is an isomorphism over $U$,
\item[---]
$K_{Y}+\Gamma=f^{*}(K_{X}+\Delta+\sum_{i=1}^{m}g^{*}c_{i})$, 
\item[---]
$K_{Y}$ is $\mathbb{Q}$-Cartier and $(Y,0)$ is semi-terminal in the sense of Fujita \cite{fujita}, and
\item[---]
$(Y,\Gamma+f^{*}g^{*}c)$ is an sdlt pair for every closed point $c\in C\setminus \{c_{1},\ldots,c_{m}\}$. 
\end{itemize}

If we are given $g\colon (X,\Delta)\to C$ as above without the assumption that the irreducible components of $X$ are normal in codimension one, we may apply Remark \ref{rem--double-cover}. 
We get a morphism $(X',\Delta')\to C$ satisfying \ref{rem--double-cover-(a)}--\ref{rem--double-cover-(d)} of Remark \ref{rem--double-cover}, then we can construct the crepant model $(Y,\Gamma)\to (X',\Delta')$ as above. 
\end{rem}

\section{Properties of semi-divisorial log terminal pairs}\label{sec5}

Our definition of sdlt pairs (Definition \ref{defn--sdlt}) is due to Koll\'ar \cite[Definition 5.19]{kollar-mmp}. The definition not only behaves well under the MMP (see Lemma \ref{lem--mmp-preserve-sdlt}) but also have a much better structure than that of slc pairs. 
In this section, we collect properties of sdlt pairs.

\begin{thm}\label{thm--sdlt-sncproperty}
Let $(X,\Delta)$ be an sdlt pair and $U\subset X$ the largest open subset such that $(U,\restriction{\Delta}{U})$ has only simple normal crossing singularities. 
Then $U$ contains all the generic points of the intersection of arbitrary finitely many irreducible components of $X$. 
\end{thm}

\begin{proof}
We pick irreducible components $X_{1},\ldots, X_{l}$ of $X$, and let $\eta$ be a generic point of $X_{1}\cap \cdots \cap X_{l}$. 
We put $U_{1}=\restriction{U}{X_{1}}$. 
It is sufficient to prove $\eta \in U_{1}$. 

By \cite[Proposition 5.20]{kollar-mmp}, $X_{1}$ is normal. 
Let $\Delta_{1}$ be an $\mathbb{R}$-divisor on $X_{1}$ such that $K_{X_{1}}+\Delta_{1}=\restriction{(K_{X}+\Delta)}{X_{1}}$. 
Then $(X_{1},\Delta_{1})$ is dlt. 
We consider $B_{i}=X_{1}\cap X_{i}$ for $2\leq i \leq l$. 
By \cite[Theorem 4.2]{fujino-van}, we see that $X_{1}\cup X_{i}$ is S$_{2}$.  
Since the morphism $X_{1}\amalg X_{i}\to X_{1}\cup X_{i}$ is not an isomorphism at all the generic points of $B_{i}$, it follows that $B_{i}$ is pure codimension one in $X_{1}$ for every $i$. 
This implies that $B_{i}$ are reduced divisors on $X_{1}$ that are components of $\lfloor \Delta_{1}\rfloor$. 
Then $\overline{\{\eta\}}$ is an lc center of $(X_{1},\Delta_{1})$ by \cite[Theorem 4.16]{kollar-mmp}. 
Thus, $\eta \in U_{1}$. 
\end{proof}

\begin{thm}\label{thm--sdlt-component-sdlt}
Let $(X,\Delta)$ be an sdlt pair and $Y \subsetneq X$ a union of irreducible components of $X$. 
Let $Y' \subset X$ be a union of irreducible components of $X$ such that $Y\cup Y'=X$ and $Y$ and $Y'$ have no common irreducible components. 
Then $Y\cap Y'$ is pure codimension one subscheme in $Y$, in other words, $Y\cap Y'$ is a Weil divisor on $Y$. 
Furthermore, if we put $\Delta_{Y}=Y\cap Y'+\restriction{\Delta}{Y}$, then $K_{Y}+\Delta_{Y}=\restriction{(K_{X}+\Delta)}{Y}$ and $(Y,\Delta_{Y})$ is an sdlt pair. 
\end{thm}

\begin{proof}
Since $X$ is demi-normal, every codimension one point of $Y$ is smooth or normal crossing singularity. 
Furthermore, \cite[Theorem 4.2]{fujino-van} shows that $Y$ is Cohen-Macaulay. 
Hence, we see that $Y$ is a demi-normal scheme. 

Let $U\subset X$ be the largest open subset such that $(U,\restriction{\Delta}{U})$ has only simple normal crossing singularities. 
Then $U$ contains all the generic points of ${\rm Supp}\Delta$, and Theorem \ref{thm--sdlt-sncproperty} shows that $U$ contains all the generic points of $Y\cap Y'$. 
So $Y\cap Y'$ is pure codimension one, and the relation $K_{Y}+\Delta_{Y}=\restriction{(K_{X}+\Delta)}{Y}$ can be proved by taking the normalization of $X$. 
We put $U_{Y}=Y\cap U$. 
By the definition of sdlt pairs, we can easily check that $(Y,\Delta_{Y})$ is slc and $U_{Y}$ satisfies the condition of sdlt pairs for $(Y,\Delta_{Y})$. 
Therefore, $(Y,\Delta_{Y})$ is sdlt.
\end{proof}

\begin{thm}\label{thm--semiter-snc}
Let $X$ be a demi-normal scheme such that $K_{X}$ is $\mathbb{Q}$-Cartier and $(X,0)$ is sdlt and semi-terminal. 
Then, for every point $x \in X$, the number of irreducible components of $X$ containing $x$ is at most two. 
In particular, for every crepant sdlt model $(Y,\Gamma)$ as in Theorem \ref{thm--crepant-sdlt} and any point $y\in Y$, the number of irreducible components of $Y$ containing $y$ is at most two. 
\end{thm}

\begin{proof}
Suppose by contradiction that there are irreducible components $X_{1}$, $X_{2}$, and $X_{3}$ of $X$ such that $X_{1} \cap X_{2} \cap X_{3}$ is not an empty set. 
By Theorem \ref{thm--sdlt-sncproperty}, we see that $D_{2}:=X_{1}\cap X_{2}$ and $D_{3}:=X_{1}\cap X_{3}$ are reduced divisors on $X_{1}$. 
Let $D$ be a $\mathbb{Q}$-divisor on $X_{1}$ such that $K_{X_{1}}+D=\restriction{K_{X}}{X_{1}}$. 
Then $D\geq D_{2}+D_{3}$. 
Since $(X_{1},D)$ is dlt, an irreducible component of $D_{2}\cap D_{3}$ is an lc center of $(X_{1},D)$ of codimension two in $X_{1}$. 
In particular, $(X_{1},D)$ is not a canonical pair.  
This contradicts the definition of semi-terminal pairs (Definition \ref{defn--semi-ter}). 
So the number of irreducible components of $X$ containing $x$ is at most two. 
\end{proof}

%%%%%%%%%%%%%%%%%%%%%
% References
%%%%%%%%%%%%%%%%%%%%%

\end{document}